\documentclass[11pt]{article}
\usepackage{amssymb,amsmath,amsfonts}
\usepackage{stmaryrd}
\usepackage{color}
\usepackage[latin1]{inputenc}
\usepackage{multirow}
\usepackage{subfigure}
\usepackage{epsfig}
\usepackage{epic}
\usepackage{pstricks}
\usepackage{pst-node}
\usepackage{pstricks-add}
\usepackage{setspace}

\def\BBox{\kern  -0.2cm\hbox{\vrule width 0.2cm height 0.2cm}}

\newtheorem{lemma}{Lemma}[section]
\newtheorem{theorem}{Theorem}[section]
\newtheorem{definition}{Definition}[section]
\newtheorem{notation}{Notation}[section]

\newtheorem{proposition}{Proposition}[section]
\newtheorem{remark}{Remark}[section]

\title{A formulation of a $(q+1,8)$--cage}

\author{ M. Abreu$^{1}$, G. Araujo-Pardo$^{2}$, C. Balbuena$^{3}$,
D. Labbate$^{1}$
\\[2ex]
$^1${\footnotesize Dipartimento di Matematica Informatica ed Economia}\\ {\footnotesize Universit\`{a} degli Studi della Basilicata}\\
{\footnotesize Viale dell'Ateneo Lucano 10, I-85100 Potenza, Italy.}\\
{\footnotesize marien.abreu@unibas.it, \, \, domenico.labbate@unibas.it}\\
\\
$^2${\footnotesize Instituto de Matem\'{a}ticas}\\
{\footnotesize Universidad Nacional Aut\'{o}noma de M\'{e}xico} \\
{\footnotesize M\'{e}xico D. F., M\'{e}xico. }\\
{\footnotesize garaujo@matem.unam.mx}\\
\\
$^3${\footnotesize Departament de Matem\'{a}tica Aplicada III}\\
{\footnotesize Universitat Polit\`{e}cnica de Catalunya}\\
{\footnotesize Campus Nord, Edifici C2, C/ Jordi Girona 1 i 3 E-08034
Barcelona, Spain.}\\
{\footnotesize  m.camino.balbuena@upc.edu}\\
}

\date{}

\begin{document}
\doublespacing
\maketitle

\begin{abstract}

Let $q\ge 2$ be a prime power.
In this note we present a formulation for obtaining the known $(q+1,8)$-cages
which has allowed us to construct small $(k,g)$--graphs for $k=q-1, q$ and   $g=7,8$. Furthermore, we also obtain smaller $(q,8)$-graphs for even prime power $q$.

\end{abstract}



\newpage

\section{Introduction}

Throughout this paper, only undirected simple graphs without loops
or multiple edges are considered. Unless otherwise stated, we
follow  the book by Bondy and Murty \cite{BM} for terminology and notation.


Let $G$ be a graph with vertex set $V=V(G)$ and edge set $E=E(G)$. The \emph{girth} of a graph $G$ is the number $g=g(G)$ of edges in a
smallest cycle. For every $v\in V$, $N_G(v)$ denotes the \emph{neighbourhood} of $v$,
i.e. the set of all vertices adjacent to $v$. The \emph{degree} of a vertex $v\in V$ is the
cardinality of    $N_G(v)$.  Let $S\subset V(G)$, then we denote by $N_G(S)=\cup_{s\in S}N_{G-S}(s)$ and by $N_G[S]=S\cup N_G(S)$.

 A graph is called
\emph{regular} if all the vertices have the same degree. A  \emph{$(k,g)$-graph} is a  $k$-regular graph with girth $g$.  Erd\H os and Sachs
 \cite{ES63}     proved the existence of  $(k,g)$-graphs
 for all values of $k$ and $g$ provided that $k \ge 2$. Since then most work carried
 out has  focused on constructing a smallest one (cf. e.g.
 \cite{AFLN06,AFLN08,AABL12,AABLS13,AABLS13EXT,AABL14,ABH10,BI73,B08,B09,B66,E96,GH08,LUW97,M99,OW81}).
A  \emph{$(k,g)$-cage} is a  $k$-regular graph with girth $g$ having the smallest possible
number of vertices. Cages have been intensely studied
since they were introduced by
Tutte \cite{T47} in 1947. More details about constructions of cages can be found in the
recent survey by Exoo and Jajcay \cite{EJ08}.

In this note we are interested in $(k,8)$--cages.
Counting the number of vertices in the
distance partition with respect to an edge yields the following lower bound
on the order of a $(k,8)$-cage:

  \begin{equation}\label{lower} n_0(k,8) = 2(1+(k-1)+(k-1)^2+(k-1)^3).\end{equation}

A  $(k,8)$--cage with $n_0(k,8)$ vertices is called a Moore \emph{$(k,8)$--graph} (cf. \cite{BM}). These graphs have been constructed as  the incidence graphs of generalized quadrangles $Q(4,q)$ and $W(q)$ \cite{B66,EJ08,PT84}, which are known to exist for $q$ a prime power and $k=q+1$
and no example is known when $k-1$ is not a prime power (cf. \cite{B97,B96,GR00,LW94}).
Since they are incidence graphs, these cages are bipartite and have diameter $4$.


In this note, we present in Definition \ref{cage} a formulation for obtaining the known $(q+1,8)$-cages  with $q\ge 2$ a prime power. Then we check in Theorem \ref{main8} that the graph $\Gamma_q$ with such a labelling is a
$(q+1,8)$--cage, for each prime power $q\ge 2$. Finally, we describe in Section \ref{appendix} the utility of this equivalent description for the known $(q+1,8)$--cages, for constructing small  $(q-1,8)$--graphs \cite{AABL14} and $(q+1,7)$--graphs \cite{AABLS13,AABLS13EXT}, that give rise to new and better upper bounds. Furthermore,  we also obtain smaller $(q,8)$-graphs for even prime power $q$.

\section{The formulation}\label{main}

We start by presenting  the following graph:

\begin{definition}\label{cage}
Let $\mathbb{F}_q$ be a finite  field with $q\ge 2$ a prime power and $\varrho$  a symbol not belonging to $\mathbb{F}_q$. Let $\Gamma_q= \Gamma_q[W_0,W_{1}]$ be a bipartite graph with vertex sets
$W_i=\mathbb{F}_q^3\cup\{(\varrho,b,c)_i, (\varrho,\varrho,c)_i: b, c \in \mathbb{F}_q \}\cup \{(\varrho, \varrho, \varrho)_i\}$, $i=0,1$, and
edge set defined as follows:
$$ \begin{array}{l}\mbox{For all } a\in \mathbb{F}_q\cup \{\varrho\}  \mbox{ and for all } b,c\in \mathbb{F}_q:\\[2ex]
N_{\Gamma_q}((a,b,c)_{1} )= \left\{\begin{array}{ll}
    \{(w,~aw+b,~a^2w+2ab+c)_{0}: w \in \mathbb{F}_q \}\cup \{(\varrho,a,c)_{0} \}  &\mbox{ if }   a\in \mathbb{F}_q;
\\[2ex]
\{( c ,b,w )_{0}: w \in \mathbb{F}_q \}\cup \{(\varrho,\varrho,c)_{0} \}  &\mbox{
if }  a= \varrho.
\end{array}\right.
\\
\mbox{}\\
N_{\Gamma_q}((\varrho,\varrho,c)_{1})= \{(\varrho,c,w)_{0}: w \in \mathbb{F}_q \}\cup
\{(\varrho,\varrho,\varrho)_{0} \}\\[1ex]
N_{\Gamma_q}((\varrho,\varrho,\varrho)_{1})= \{(\varrho,\varrho,w)_{0}: w \in \mathbb{F}_q\} \cup
\{(\varrho,\varrho,\varrho)_{0} \}.
\end{array}
$$
Or equivalently
$$ \begin{array}{l}\mbox{For all } i\in \mathbb{F}_q\cup \{\varrho\}  \mbox{ and for all }  j,k\in \mathbb{F}_q:\\[2ex]
N_{\Gamma_q}((i,j,k)_{0} )= \left\{\begin{array}{ll}
    \{(w,~j-wi,~w^2i-2wj+k)_{1}: w \in \mathbb{F}_q \}\cup \{(\varrho,j,i)_1 \}  &\mbox{ if }   i\in \mathbb{F}_q;
\\[2ex]
\{( j ,w,k )_{1}: w \in \mathbb{F}_q \}\cup \{(\varrho,\varrho,j)_{1} \}  &\mbox{ if }  i= \varrho.
\end{array}\right.
\\
\mbox{}\\
N_{\Gamma_q}((\varrho,\varrho,k)_{0})= \{(\varrho,w,k)_{1}: w \in \mathbb{F}_q \}\cup
\{(\varrho,\varrho,\varrho)_{1} \}; \\[1ex]
N_{\Gamma_q}((\varrho,\varrho,\varrho)_{0})= \{(\varrho,\varrho,w)_{1}: w \in \mathbb{F}_q\} \cup
\{(\varrho,\varrho,\varrho)_{1} \}.
\end{array}
$$
\end{definition}

\noindent Note that $\varrho$ is just a symbol not belonging to $\mathbb{F}_q$ and no arithmetical operation  will be performed  with it.

Next, we will make use of the following induced subgraph $B_q$ of $\Gamma_q$.

\begin{notation}\label{Bq}
Let $B_q=B_q[V_0,V_1]$ be a bipartite graph with vertex set  $V_i=\mathbb{F}_q^3$, $i=0,1$,
and edge set $E(B_q)$ defined as follows:
$$
\mbox{For all }   a, b,c\in \mathbb{F}_q  : N_{B_q}((a,b,c)_{1} )=
    \{(j,~aj+b,~a^2j+2ab+c)_{0}: j \in \mathbb{F}_q \} .
$$
\end{notation}

In order to proceed, we need the following definition of a $q$-regular bipartite graph $H_q$  introduced by Lazebnik, Ustimenko and Woldar \cite{LUW97}.

\begin{definition}\cite{LUW97} \label{Hq}
Let  $\mathbb{F}_q$ be a finite field with $q\ge 2$.
Let $H_{q }=H_q[U_0,U_1]$ be a bipartite graph with vertex set $U_r=\mathbb{F}_q^3$, $r=0,1$, and edge set $E(H_q)$ defined as follows:
$$
\mbox{For all }   a, b,c\in \mathbb{F}_q  : N_{H_q}((a,b,c)_{1} )=
    \{(w,~aw+b,~a^2w +c)_{0}: w \in \mathbb{F}_q \} .
 $$
\end{definition}

Lazebnik, Ustimenko and Woldar proved that $H_q$ given in Definition \ref{Hq} is $q$-regular, bipartite, of girth $8$ and order $2q^3$.

\begin{lemma}\label{HqisoBq}
The graph $B_q$
is isomorphic to the graph $H_q$.
\end{lemma}

\begin{proof}
Let $H_q$ be the bipartite graph from Definition \ref{Hq}. Since the map $\sigma: B_q \to H_q$ defined by
$\sigma((a,b,c)_{1})=(a,b,2ab+c)_{1}$ and $\sigma((x,y,z)_{0})=(x,y,z)_{0}$ is an isomorphism,
the result holds.
\end{proof}

In what follows, we will obtain the graph $\Gamma_q$ from the graph $B_q$ adding some new vertices and edges.
 We need a preliminary lemma.

\begin{lemma}\label{claim01}
Let $B_q$ be the graph from Notation \ref{Bq}.
For any given $a\in \mathbb{F}_q $,  the vertices in the set
$\{(a,b,c)_{1}:  b,c \in \mathbb{F}_q \}$ are mutually at distance
at least four.
And, for any given $i\in \mathbb{F}_q $, the vertices in the set
$\{(i,j, k)_{0}:  j,k  \in \mathbb{F}_q \}$ are mutually at distance at least four.
\end{lemma}

\begin{proof}
Suppose that
there exists  a path of length two $(a,b,c )_{1}  (w,j,k)_{0}  (a,b',c' )_{1}$ in $B_q$. By Definition \ref{cage}, $j=aw+b=aw+b'$ and $k=a^2w+2ab+c=a^2w+2ab'+c'$. From both equation we get
$b=b'$ and $c=c'$   which implies that $(a,b,c )_{1}=(a,b',c' )_{1}$ contradicting that the path has length two. Similarly suppose
that there exists a path of length two $(i,j, k)_{0}  (a,b,c )_{1}  (i,j', k')_{0}$.
Reasoning similarly, we obtain   $j=ai+b=j'$, and    $c=a^2i-2aj+k=a^2i-2aj'+k'$ yielding $(i,j, k)_{0}=(i,j', k')_{0}$
which is a contradiction.
\end{proof}


Figure \ref{spanning} shows a spanning tree of $\Gamma_q$ with the vertices  labelled according to Definition \ref{cage}.

\begin{figure}[h]
  \centering
  \includegraphics[width=16cm]{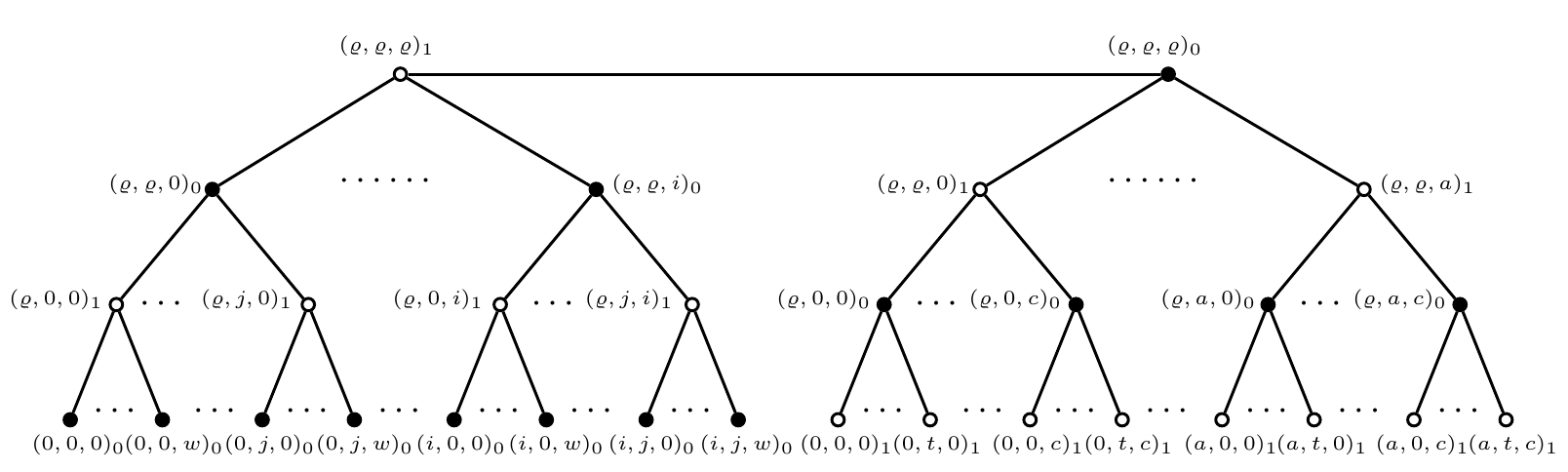}\\
  \caption{Spanning tree of $\Gamma_q$.}\label{spanning}
\end{figure}


\begin{theorem}\label{main8}
The graph $\Gamma_q$ given in Definition \ref{cage}  is a  Moore $(q+1,8)$-graph   for each prime power $q\ge 2$.
\end{theorem}

\begin{proof}
As a consequence of Lemma \ref{claim01}, we obtain the following claim.

\noindent\emph{Claim 1: For all $x,y\in\mathbb{F}_q $, the $q$
vertices of the set $ \{(x,y,j)_{0}: j\in
\mathbb{F}_q \}$ are mutually at
distance at least 6 in $B_q$.}

\noindent {\it Proof:}
By Lemma \ref{claim01},   the
  $q$ vertices $\{(x,y,j )_{0}:j\in
\mathbb{F}_q\}$   are mutually at distance at least 4. Suppose by contradiction
that $B_q$ contains the following path of length four:
$$(x,y,j)_{0}  ~(a,b,c )_{1}~  (x',y',j')_{0} ~ (a',b',c' )_{1}  ~(x,y,j'')_{0}, \mbox{ for some } j''\ne j.$$
Then $y=  ax +b=  a' x+b'$ and $y'=   ax'+b  =  a'x'+b' $.
It follows that $  (a-a')(x-x') =0$, which is a contradiction since  by Lemma \ref{claim01}  $a\ne a'$ and $x\ne x'$.  \quad $\Box$

Let $B'_{q}=B'_q[V_{0},V_{1}']$ be the bipartite graph obtained from $B_q=B_q[V_{0},V_{1}]$ by adding $q^2$ new vertices to $V_{1}$ labeled $(\varrho,b,c)_{1}$,  $b,c\in \mathbb{F}_q$ (i.e., $V_{1}'=V_1\cup  \{(\varrho,b,c)_{1}: b,c\in \mathbb{F}_q\}$), and new edges $N_{B'_{q}}((\varrho,b,c)_{1})= \{(c, b ,j )_{0}: j \in \mathbb{F}_q \}$ (see Figure \ref{spanning}).
Then $B'_{q}$ has $|V'_{1}|+|V_{0}|=2q^3+q^2$ vertices such that every vertex of $V_{0}$ has degree $q+1$ and every vertex of $V_{1}'$ has still degree $q$. Note that the girth of $B'_{q}$ is 8 by  Claim 1. Further, Lemma \ref{claim01}    partially holds in $B'_{q}$. We write this fact in the following claim.

\noindent\emph{Claim 2: For any given $a\in \mathbb{F}_q\cup\{\varrho\} $,  the vertices of the set
$\{(a,b,c )_{1}:  b,c \in \mathbb{F}_q \}$ are mutually at distance
at least four in $B'_{q}$.}

  \bigskip

\noindent\emph{Claim 3: For all $a\in\mathbb{F}_q\cup \{\varrho\} $ and for all $c\in\mathbb{F}_q $, the $q$
vertices of the set $ \{(a,t,c)_{1}: t\in
\mathbb{F}_q \}$ are mutually at
distance at least 6 in $B'_{q}$.}

\noindent {\it Proof:}
By Claim 2, for all $a\in\mathbb{F}_q \cup \{\varrho\}  $ the $q$ vertices of
 $ \{(a,t, c)_{1}: t\in
\mathbb{F}_q\}$  are  mutually at distance at least 4   in  $B'_{q}$. Suppose
that there exists in $B'_{q}$ the following path of length four:
$$(a,t, c)_{1}  ~(x,y,z)_{0} ~ (a',t',c')_{1} ~ (x',y',z')_{0} ~ (a,t'',c)_{1}, \mbox{ for some } t''\ne t.$$

 \noindent If $a=\varrho$, then $x=x'=c$, $y=t$, $y'=t''$ and $a'\ne \varrho$ by Claim 2.  Then $y=a' x+t'=a' x'+t'=y'$ yielding that $t=t''$ which is a contradiction. Therefore
 $a\ne \varrho$.   If $a'=\varrho$, then $x=x'=c'$ and $y=y'=t'$. Thus $y=a  x+t=ax'+ t''=y'$ yielding that $t=t''$ which is a contradiction. Hence we may assume that $a'\ne \varrho$ and $a\ne a'$ by Claim 2. In this case we have:
  $$
    \begin{array}{ll}
               y=ax+ t &= a'x+t' ;\\
           y'= ax'+t''      &=  a'x'+t'  ;\end{array}
    \begin{array}{ll}
              z=  a^2x +    2a t+c & =  a'^2x +  2a't' +c';\\
               z'=   a^2x' +    2 at''+c & =  a'^2x' +  2a't'+c'.
              \end{array}
              $$
 Hence
 \begin{eqnarray}
 \label{dosa} (a-a')(x-x')&=t''-t;\\
 \label{dosb} (a^2-a'^2)(x-x')&= 2a(t''- t).
\end{eqnarray}
If $q$ is even, (\ref{dosb}) leads to  $x=x'$ and (\ref{dosa}) leads to $t''=t$   which is a contradiction with our assumption.
Thus assume $q$ is odd. If $a+a'=0$, then  (\ref{dosb}) gives
              $2a(t''-t)=0$, so that $ a=0$ yielding that $a'=0$ (because $a+a'=0$)
              which is again a contradiction.
If $a+a'\ne 0$, multiplying  equation  (\ref{dosa}) by $a+a'$  and subtracting both equations we obtain  $(2a-(a+a'))(t''-t)=0$. Then   $a=a'$ because $t''\ne t$, which is a
              contradiction to Claim 2.
  Therefore, Claim 3 holds.  \quad $\Box$

Let $B''_{q}=B''_{q}[V'_{0},V_{1}']$ be the graph obtained from $B'_{q}=B'_{q}[V_0,V'_1]$ by adding $q^2+q$ new vertices to $V_{0}$ labeled $(\varrho,a,c)_{0}$, $a\in\mathbb{F}_q\cup \{\varrho\}$, $c\in \mathbb{F}_q$,   and new edges $N_{B''_{q}}((\varrho,a,c)_{0})= \{(a, t ,c )_{1}: t \in \mathbb{F}_q \}$ (see Figure \ref{spanning}).
Then $B''_{q}$ has $|V'_{1}|+|V_{0}'|=2q^3+2q^2+q$ vertices such that every vertex  has degree $q+1$ except the new added vertices which have degree $q$. Moreover the girth of $B''_{q}$ is 8 by  Claim 3.

\emph{Claim 4: For all $a\in\mathbb{F}_q\cup \{\varrho\} $, the $q$
vertices of the set $ \{(\varrho,a,j)_{0}: j\in
\mathbb{F}_q \}$ are mutually at
distance at least 6 in $B''_{q}$.}

\noindent {\it Proof:}
Clearly   these $q$ vertices are  mutually at distance at least 4   in  $B''_{q}$. Suppose
that there exists in $B''_{q}$ the following path of length four:
$$(\varrho,a, j)_{0}~ (a,b,j)_{1}~ (x,y,z)_{0}~ (a,b',j')_{1}~ (\varrho,a,j')_{0}, \mbox{ for some } j'\ne j.$$
If $a=\varrho$ then $x=j=j'$ which is a contradiction.  Therefore $a\ne \varrho$. In this case $y=ax+b=ax+b'$ which implies that $b=b'$. Hence $z=a^2x+2ab+j=a^2x+2ab'+j'$ yielding that $j=j'$ which is again a contradiction. \quad $\Box$

Let $B'''_{q}=B'''_{q}[V_{0}',V''_{1}]$ be the graph obtained from $B''_{q}$ by adding $q+1$ new vertices to $V_{1}'$ labeled $(\varrho,\varrho,a)_{1}$,  $a\in \mathbb{F}_q\cup \{\varrho\}$,    and new edges $N_{B'''_{q}}(\varrho,\varrho,a)_{1}= \{(\varrho,  a,c )_{0}: c\in \mathbb{F}_q \}$, see Figure \ref{spanning}.
Then $B'''_{q}$ has $|V''_{1}|+|V_{0}'|=2q^3+2q^2+2q+1$ vertices such that every vertex  has degree $q+1$ except the new added vertices which have degree $q$. Moreover the girth of $B'''_{q}$ is 8 by  Claim 4 and clearly   these $q+1$ new vertices are mutually at distance 6. Finally,    the   Moore $(q+1,8)$-graph $\Gamma_q$ is obtained by adding to $B'''_{q}$ another new vertex labeled $(\varrho,\varrho,\varrho)_{0}$ and edges  $N_{\Gamma_q}((\varrho,\varrho,\varrho)_{0})= \{(\varrho,\varrho,i)_{1}: i\in \mathbb{F}_q\cup \{\varrho\}\}$.
\end{proof}

\begin{remark}\label{equiv}
A coordinatization of classical generalized quadrangles $Q(4,q)$ and $W(q)$  in four dimensions are discussed in  \cite{M98,P70,U90}.
The formulation of a Moore $(q+1,8)$-graph given in Theorem \ref{main8} in three dimensions is equivalent to this coordinatization.
\end{remark}


\section{Applications}\label{appendix}

In this section we overview results that we have obtained for girth $7$ in \cite{AABLS13, AABLS13EXT} and for girth $8$ in \cite{AABL14}  using the labelling from Definition \ref{cage}. Moreover, we also use the labelling to construct small $(q,8)$--graphs, for $q$ even,  which we match the bound on the order  obtained by \cite{GH08}.

For girth $7$, we have obtained the following results:

\begin{theorem}\label{main1}(\cite[Theorem 3.5]{AABLS13}, \cite[Theorem 2.4]{AABLS13EXT})
Let $q\ge 4$ be an even prime power. Then, there is a $(q+1)$-regular graph of girth 7 and order $2q^3+q^2+2q$.
\end{theorem}

\begin{theorem}\label{main2}(\cite[Theorem 4.7]{AABLS13}, \cite[Theorem 3.4]{AABLS13EXT})
Let $q\ge 5$ be an odd prime power. Then, there is a $(q+1)$-regular graph of girth 7 and order $2q^3+2q^2-q+1$.
\end{theorem}

A subset $U\subset V(G)$ is said to be \emph{perfect dominating set} of $G$ if for each vertex $x\in V(G)\setminus U$, $|N_G(x)\cap U|=1$ (cf. \cite{HHS98}).  Note that if $G$ is a $k$-regular graph and $U$ is a perfect dominating set of $G$ then $G-U$ is clearly a $(k-1)$-regular graph.
In \cite{AABL14}, using the labelling from Definition \ref{cage}, we have obtained perfect dominating sets of $(q+1,8)$--cages and $(q,8)$--graphs, for $q$ a prime power, that give rise to the following two theorems:

\begin{theorem}\label{main3}
Let $q\ge 4$ be a  prime power. Then, there is a $q$-regular graph of girth 8 and order $2q(q^2-2)$.
\end{theorem}

\begin{theorem}\label{main4}
Let $q\ge 4$ be a  prime power. Then, there is a $(q-1)$-regular graph of girth 8 and order $2q(q-1)^2$.
\end{theorem}

The labelling from Definition \ref{cage} also allows us to improve the result in Theorem \ref{main3} for even $q \ge 4$.

\begin{proposition}\label{perfect2}Let $q\ge 4$ be an even  prime power  and   $\Gamma_q=\Gamma_q[V_{0},V_{1}]$ the Moore $(q+1, 8)$-graph given in Definition \ref{cage}.
 Let $Q=\{(\varrho,j,0)_{0}: j\in \mathbb{F}_q\}\cup \{(\varrho,\varrho,0)_{0}\}$ and $S=\{(u,u,1+u+u^2)_{1}: u\in \mathbb{F}_q\}\cup \{(\varrho,1,1)_{1}\}$. Then the set
 $$ N_{\Gamma_q}[Q]\cup \left(\bigcap_{a\in Q}N_{\Gamma_q}^2(a)\right)\cup N_{\Gamma_q}[S]\cup \left(\bigcap_{b\in S}N_{\Gamma_q}^2(b)\right)$$
 is a
  perfect dominating set of  $\Gamma_q$ of cardinality $ 2(q^2+4q+3)$ (see Figure \ref{estruc2}).

  \end{proposition}
\begin{figure}[h]
  \centering
  \includegraphics[width=\textwidth]{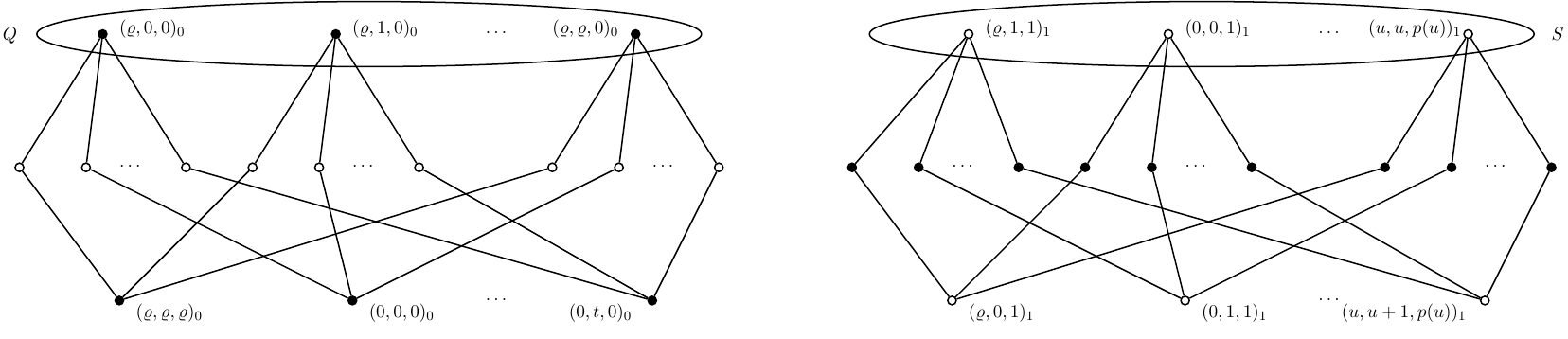}\\
  \caption{Perfect dominating set   of Theorem
\ref{perfect2}.}\label{estruc2}
\end{figure}

\begin{proof} First assume that $q\ge 8$.
By Definition \ref{cage}, we have for all element of $ Q$:

$N_{\Gamma_q}((\varrho,j,0)_{0})=\{(j,t,0)_{1}:t\in \mathbb{F}_q\}\cup \{(\varrho,\varrho,j)_{1}\}$;

$N_{\Gamma_q}((\varrho,\varrho,0)_{0})=\{(\varrho,t,0)_{1}:t\in \mathbb{F}_q\}\cup \{(\varrho,\varrho,\varrho)_{1}\}$.

Then $(\varrho,\varrho,\varrho)_0\in N_{\Gamma_q}^2((\varrho,j,0)_{0}))\cap  N_{\Gamma_q}^2((\varrho,\varrho,0)_{0}))$ for all $j\in \mathbb{F}_q$. Moreover, since $q$ is even, $2tj=0$ and $N_{\Gamma_q}((j,t,0)_{1})=\{(w,jw+t,j^2w)_0:w\in \mathbb{F}_q\}\cup \{(\varrho,j,0)_{0}\}$. Thus, for all $j_1,j_2\in \mathbb{F}_q$, $j_1\ne j_2$, we have $(w,j_1w+t_1,j_1^2w)_0=(w,j_2w+t_2,j_2^2w)_0$ if and only if $w=0$ and $t_1=t_2$. Let $I_Q=\bigcap_{a\in Q}N_{\Gamma_q}^2(a)$.  Then $I_Q =\{(0,t,0)_{0}: t\in \mathbb{F}_q\}\cup \{(\varrho,\varrho,\varrho)_{0}\}$ (see left side of Figure \ref{estruc2}) implying that $|N_{\Gamma_q}[Q]|+|I_Q|=(q+1)^2+2(q+1)$.

Let $p(u)=1+u+u^2$ and observe that $j=-j$ for all $j\in \mathbb{F}_q$ since $q$ is even.
By Definition \ref{cage}, we have for all element of $S$:

 $N_{\Gamma_q}  ( (u,u,p(u) )_{1}  \}=\{(a,ua+u,u^2a+p(u))_0: a\in \mathbb{F}_q\} \cup \{(\varrho,u,p(u))_0\}$;

  $N_{\Gamma_q}((\varrho,1,1)_{1})=\{(1,1,a)_{0}:a\in \mathbb{F}_q\}\cup \{(\varrho,\varrho,1)_{0}\}$;

  $N_{\Gamma_q}((a,ua+u,u^2a+p(u))_0)=\{ (w,ua+u+wa,w^2a+u^2a+p(u) )_{1}  :w\in \mathbb{F}_q \}\cup \{   (\varrho,ua+u,a)_{1}\} $.

   $N_{\Gamma_q}((\varrho,u,p(u))_{0})=\{(u,y,p(u))_{1}:y\in \mathbb{F}_q\}\cup \{(\varrho,\varrho,u)_{1}\}$;

  Then $(\varrho,0,1)_1\in N_{\Gamma_q}((\varrho,\varrho,1)_{0}))\cap  N_{\Gamma_q}((1,0,1+u)_{0}))$ for all $u\in \mathbb{F}_q$. Let $I_S=\bigcap_{b\in S}N_{\Gamma_q}^2(b)$, we have $(\varrho,0,1)_1\in I_S $ (see right side of Figure \ref{estruc2}). Moreover, two vertices $(a_1, a_1u_1+u_1, a_1u_1^2+p(u_1))_0$ and $(a_2, a_2u_2+u_2, a_2u_2^2+p(u_2))_0$ for $u_1,u_2,a_1,a_2\in \mathbb{F}_q$ such that $u_1\ne u_2$ and $a_1,a_2\ne 1$ have a common neighbor $(w, x , y)_1$ for some $w,x,y\in \mathbb{F}_q$ if and only if
  \begin{equation}\label{x} x= u_1a_1+u_1+wa_1 =u_2a_2+u_2+wa_2 \end{equation} and
  \begin{equation}\label{y} y=u_1^2a_1+w^2a_1+p(u_1) =u_2^2a_2+w^2a_2+p(u_2) .\end{equation}
But (\ref{x}) holds if and only if
   \begin{equation}\label{xx} x+w=(u_1+w)(a_1+1) =(u_2+w)(a_2+1).\end{equation}   Moreover, (\ref{y}) holds if and only if
   \begin{equation} \label{yy} y+w^2+1=(u_1^2+w^2)(a_1+1)+u_1 =(u_2^2+w^2)(a_2+1)+u_2. \end{equation}
Since $q$ is even $(u_i^2+w^2)=(u_i+w)^2$. Therefore
  multiplying (\ref{xx}) by $u_1+w$ and combining it with (\ref{yy}) we get $u_1=(u_2+w)(a_2+1)(u_1+u_2)+u_2$. But $u_1\ne u_2$ gives that $(u_2+w)(a_2+1)=1$ as well as $(u_1+w)(a_1+1)=1$ by (\ref{xx}). Hence,   $x=1+w$ by (\ref{xx})  and  $y=1+w+w^2=p(w)$ by (\ref{yy}).
We conclude that $I_S=\{(u,1+u,p(u))_1:u\in \mathbb{F}_q\}\cup \{(\varrho,0,1)_1\}$
since $(u,1+u,p(u))_1\in N_{\Gamma_q}((\varrho,u,p(u))_{0})$. Thus, $|N_{\Gamma_q}[S]|+|I_S|=(q+1)^2+2(q+1)$.

Let $D_Q=N_{\Gamma_q}[Q]\cup I_Q$ and $D_S=N_{\Gamma_q}[S]\cup I_S$. We get that
$|D_Q\cup D_S|=2(q+1)^2+4(q+1)=2(q^2+4q+3)$ since $D_Q$ and $D_S$ are vertex disjoint for $q\ge 8$.

Let us show that $D_Q\cup D_S$ is a perfect dominating set of $\Gamma_q$.

First, we check that  by Definition \ref{cage}, there is a matching joining each vertex of $  D_Q\cap V_1$ with one vertex in $D_S\cap V_0$. We have:

\begin{itemize}
\item[-] For all $u\in\mathbb{F}_q$, $(\varrho, \varrho, u)_1\in D_Q$ is adjacent to $ (\varrho,   u, p(u))_0\in D_S$, and $(\varrho, \varrho, \varrho)_1\in D_Q$ is adjacent to $ (\varrho,  \varrho, 1)_0\in D_S$.

\item[-] For all $t\in\mathbb{F}_q$, $(\varrho, t, 0)_1\in D_Q$ is adjacent to $ (0, t, p(t))_0\in D_S$.

\item[-] For all $a\in\mathbb{F}_q$, $(a, a, 0)_1\in D_Q$ is adjacent to $ (1, 0, a^2)_0\in N_{\Gamma_q}(\varrho, 0,1)_1\subset D_S$.

\item[-] For all $a\in\mathbb{F}_q$, $(a, a+1, 0)_1\in D_Q$ is adjacent to $ (1, 1, a^2)_0\in N_{\Gamma_q}(\varrho, 1,1)_1\subset D_S$.

\item[-] For all $a,t\in\mathbb{F}_q$, $  t\ne 0,1$, $(a, a+t, 0)_1\in D_Q$ is adjacent to $ (x, ax+a+t, a^2x)_0\in N_{\Gamma_q}(u, u,p(u))_1\subset D_S$ where $x,u\in\mathbb{F}_q $  are the solution of the system
    $$\left.\begin{array}{cl}   a(x+1)+t &=u(x+1)\\ a^2x &=u^2(x+1)+u+1 \end{array}\right\} \Leftrightarrow
    \left\{\begin{array}{cl}   (x+1)(u+a)  &=t,\\ (a^2+u^2)(x+1)+u+1 &=a^2.\end{array}\right. $$
 Since $q$ is even $(a^2+u^2)=(a+u)^2$, yielding that $t(u+a)+u+1=a^2$, or equivalently $(t+1)(u+a)=p(a)$. Consequently $u=a+(t+1)^{-1}p(a)$ and    $x=1+t(t+1)p(a)^{-1}$.
\end{itemize}

 Let $H$ be the subgraph of $\Gamma_q$ induced by $   D_Q\cup D_S$. The existence of the matching joining each vertex of $  D_Q\cap V_1$ with one vertex in $D_S\cap V_0$ allows us to conclude that  the vertices of $H$ have degree 3 or $q+1$ in $H$, and the diameter of $H$ is 5. Moreover since
  the girth is 8 we obtain $|N_{\Gamma_q}((D_Q\cup D_S))\cap (V(\Gamma_q)\setminus (D_Q\cup D_S)) |=2(q-2)(q+1)^2=2(q^3-3q-2)=|V(\Gamma_q)\setminus (D_Q\cup D_S) |$
yielding that $|N_{\Gamma_q}(v)\cap (D_Q\cup D_S)|= 1$ for all $v\in V(\Gamma_q)\setminus (D_Q\cup D_S)$. Therefore $D_Q\cup D_S$ is a perfect dominating set of $\Gamma_q$ and the result holds for $q\ge 8$.

Finally, for $q=4$,  note that $p(x)=1+x+x^2\in \{0,1\}$ for all $x\in \mathbb{F}_4$ which implies that $N_{\Gamma_4}[S]\cup \left(\bigcap_{b\in S}N_{\Gamma_4}^2(b)\right)$ and $N_{\Gamma_4}[Q]\cup \left(\bigcap_{b\in Q}N_{\Gamma_4}^2(b)\right)$ defined in Proposition \ref{perfect2} are not vertex disjoint. Take $x\in \mathbb{F}_4\setminus\{0,1\}$, and let $Q'=\{(\varrho,j,x)_{0}: j\in \mathbb{F}_4\}\cup \{(\varrho,\varrho,0)_{0}\}$ and $S=\{(u,u,1+u+u^2)_{1}: u\in \mathbb{F}_4\}\cup \{(\varrho,1,1)_{1}\}$. Then the set
 $$ N_{\Gamma_4}[Q']\cup \left(\bigcap_{a\in Q'}N_{\Gamma_4}^2(a)\right)\cup N_{\Gamma_4}[S]\cup \left(\bigcap_{b\in S}N_{\Gamma_4}^2(b)\right)$$
is a perfect dominating set of the Moore $(5,8)$-graph of cardinality 70. Consequently by removing this set we obtain a 4-regular graph of girth 8 with 100 vertices.
\end{proof}


\begin{theorem}\label{regq}
Let $q\ge 4$ be an even  prime power  and  $\Gamma_q=\Gamma_q[V_{0},V_{1}]$ be the Moore $(q+1, 8)$-graph given in Definition \ref{cage}.
 Then there is
a $q$-regular graph of girth 8 and order $2(q^3-3q-2)$.
\end{theorem}
\begin{proof}
The result is obtained by removing from $\Gamma_q$ the perfect dominating set  found in Proposition \ref{perfect2}.
\end{proof}


In \cite{GH08} G\`{a}cs and H\'{e}ger    have obtained   $(q,8)$--bipartite graphs  on $2q(q^2-2)$
vertices if $q$ is odd, or on $2(q^3-3q-2)$ vertices if $q$ is
even,   using a regular point pair in a classical generalized quadrangle. Note
that in Theorem \ref{regq} we explicitly obtain  $(q,8)$--bipartite graphs of the same cardinality
using Definition \ref{cage}.
In  \cite{ABH10}  $(k,8)$--regular balanced
bipartite graphs for all prime  power   $q$  such that $3\le
k\le q$ of order $2k ( q^2-1)$ have been obtained  as   subgraphs
of the incidence graph of a generalized quadrangle.
In  \cite{B09} this result has been improved by constructing $(k,8)$-regular
balanced bipartite graphs of order $2q (kq -1)$.

\section*{Acknowledgements}
This research  was supported by the Ministerio de Educaci\'{o}n y Ciencia,
Spain, the European Regional Development Fund (ERDF) under project
MTM2011-28800-C02-02; by the Catalonian Government under
project 1298 SGR2009; by CONACyT-M\'{e}xico under project 57371;
by PAPIIT-M\'{e}xico under project 104609-3; by the Italian Ministry MIUR
and carried out within the activity of INdAM-GNSAGA.


%
%
%

\end{document}